\documentclass[a4paper]{amsart}
\usepackage{hyperref, aliascnt}

\newcommand{\andSep}{\,\,\,\text{ and }\,\,\,}
\newcommand{\ca}{$C^*$-algebra}

\DeclareMathOperator{\linspan}{span}
\DeclareMathOperator{\interior}{int_D}
\DeclareMathOperator{\exterior}{cl_D}
\newcommand{\Imin}{I_{\mathrm{min}}}
\newcommand{\Imax}{I_{\mathrm{max}}}

\def\today{\number\day\space\ifcase\month\or   January\or February\or
   March\or April\or May\or June\or   July\or August\or September\or
   October\or November\or December\fi\   \number\year}

\newtheorem{lma}{Lemma}[section]

\newaliascnt{thmCt}{lma}
\newtheorem{thm}[thmCt]{Theorem}
\aliascntresetthe{thmCt}

\newaliascnt{corCt}{lma}
\newtheorem{cor}[corCt]{Corollary}
\aliascntresetthe{corCt}

\newaliascnt{prpCt}{lma}
\newtheorem{prp}[prpCt]{Proposition}
\aliascntresetthe{prpCt}

\theoremstyle{definition}

\newaliascnt{dfnCt}{lma}

\aliascntresetthe{dfnCt}

\newaliascnt{rmkCt}{lma}
\newtheorem{rmk}[rmkCt]{Remark}
\aliascntresetthe{rmkCt}

\newaliascnt{exaCt}{lma}
\newtheorem{exa}[exaCt]{Example}
\aliascntresetthe{exaCt}

\newaliascnt{qstCt}{lma}
\newtheorem{qst}[qstCt]{Question}
\aliascntresetthe{qstCt}

\newaliascnt{pgrCt}{lma}
\newtheorem{pgr}[pgrCt]{}
\aliascntresetthe{pgrCt}

\newcounter{theoremintro}

\newaliascnt{thmIntroCt}{theoremintro}
\newtheorem{thmIntro}[thmIntroCt]{Theorem}
\aliascntresetthe{thmIntroCt}

\newaliascnt{prpIntroCt}{theoremintro}

\aliascntresetthe{prpIntroCt}

\newaliascnt{corIntroCt}{theoremintro}
\newtheorem{corIntro}[corIntroCt]{Corollary}
\aliascntresetthe{corIntroCt}

\newaliascnt{qstIntroCt}{theoremintro}

\aliascntresetthe{qstIntroCt}

\title{Prime and semiprime Lie ideals in C*-algebras}

\author[Eusebio Gardella]{Eusebio Gardella}
\address{Eusebio Gardella,
Department of Mathematical Sciences, Chalmers University of
Technology and University of Gothenburg, Gothenburg SE-412 96, Sweden.}
\email{gardella@chalmers.se}
\urladdr{www.math.chalmers.se/~gardella}

\author{Kan Kitamura}
\address{Kan Kitamura,
Department of Mathematics, College of Science, Rikkyo University, 
3-34-1, Nishi-Ikebukuro, Toshima-ku, Tokyo, 171-8501, Japan.}
\address{Center for Interdisciplinary Theoretical and Mathematical Sciences, 
RIKEN, 2-1 Hirosawa, Wako, Saitama 351-0198 Japan.}
\email{kan.kitamura@rikkyo.ac.jp}

\author{Hannes Thiel}
\address{Hannes~Thiel, 
Department of Mathematical Sciences, Chalmers University of Technology and University of
Gothenburg, Gothenburg SE-412 96, Sweden.}
\email{hannes.thiel@chalmers.se}
\urladdr{www.hannesthiel.org}

\thanks{
The first named author was partially supported by the Swedish Research Council Grant 2021-04561.
The second named author was partially supported by JSPS KAKENHI Grant Number JP25K17272 and RIKEN Special Postdoctoral Researcher Program.
The third named author was partially supported by the Knut and Alice Wallenberg Foundation (KAW 2021.0140) and by the Swedish Research Council (VR) project grant 2024-04200.
}

\subjclass[2020]%
{Primary
46L05. 
Secondary
16N60, 
16W10, 
17B60. 
}
\date{\today}

\begin{document}

\begin{abstract}
Using the theory of Dixmier ideals developed in previous
work, we show that every semiprime Lie ideal in a \ca{} arises as the full normalizer subspace of a semiprime two-sided ideal.
This leads to a concise description of all semiprime Lie ideals in terms of semiprime two-sided ideals, and an analogous description of prime Lie ideals in terms of prime two-sided ideals.

For unital \ca{s} without characters, we obtain a natural bijection between (semi)prime two-sided ideals and (semi)prime Lie ideals, and it follows that a Lie ideal is fully noncentral if and only if it is not contained in any prime Lie ideal.
\end{abstract}	

\maketitle

\section{Introduction}

Every \ca{} $A$ carries a natural Lie algebra structure with the Lie product of operators $a,b \in A$ given by their additive commutator $[a,b] := ab-ba$.
A \emph{Lie ideal} in $A$ is a linear subspace $L \subseteq A$ such that $[A,L] \subseteq L$.
Obvious examples of Lie ideals are all of $A$, its center $Z(A)$, and the commutator subspace $[A,A]$.
Here, we follow the usual convention that $[E,F]$ and $EF$ denote the linear subspaces of $A$ generated by $\{[a,b] : a \in E, b \in F\}$ and $\{ab : a \in E, b \in F \}$, respectively.
We emphasize that we study Lie ideals that are not necessarily closed.

We refer to Section~2 in \cite{Mar10ProjCommutLieIdls} for an introduction to the extensive study of Lie ideals in \ca{s}.
Let us mention just a few results in the case that~$A$ is a simple, unital \ca{}.
First, as a consequence of Herstein's description of Lie ideals in simple rings \cite{Her55LieJordanSimpleRing}, a subspace of $A$ is a Lie ideal if and only if it is either contained in the center $Z(A)$ or it contains $[A,A]$.
If $A$ admits no tracial states (for example, if $A$ is a simple, purely infinite \ca{}), then $A=[A,A]$ by \cite{Pop02FiniteSumsCommutators} and it follows that $A$ has precisely three Lie ideals: $\{0\}$, $Z(A)$ and $A$.
If $A$ admits a unique tracial state, then $\overline{[A,A]}$ is a closed subspace of codimension one by \cite{CunPed79EquivTraces}, and $A$ has precisely four closed Lie ideals: $\{0\}$, $Z(A)$, $\overline{[A,A]}$ and $A$.
If $A$ is also pure (\cite{AntPerThiVil24arX:PureCAlgs}) and exact, then $[A,A]$ is closed by \cite{NgRob16CommutatorsPureCa}, and it follows that the only Lie ideals are $\{0\}$, $Z(A)$, $[A,A]$ and $A$.
This applies for example to UHF-algebras, irrational rotation algebras, the Jiang-Su algebra, as well as many reduced group \ca{s} \cite{AmrGaoElaPat25StrCompRedGpCAlgs}.

When studying Lie ideals in non-simple \ca{s}, it is natural that the lattice of two-sided ideals enters the picture.
In this paper, by a two-sided ideal in a \ca{} $A$ we mean a linear subspace $I \subseteq A$ such that $ax,xa \in I$ for all $a \in A$ and $x \in I$.
We warn the reader that this terminology deviates from that in \cite{GarKitThi23arX:SemiprimeIdls}, where we used the ring-theoretic notion (an additive subgroup $I \subseteq A$ such that $ax,xa \in I$ for all $a \in A$ and $x \in I$).
A ring-theoretic ideal is automatically a linear subspace if $A$ is unital or if the ideal is semiprime (\cite[Corollary~5.5]{GarKitThi23arX:SemiprimeIdls}), but not in general \cite[Examples~II.5.2.1]{Bla06OpAlgs}.
A systematic investigation of the connection between Lie ideals and two-sided ideals in a \ca{} was initiated in \cite{BreKisShu08LieIdeals}.

Primeness and semiprimeness of ideals in Lie algebras are defined analogous to ideals in associative rings;
see, for example, \cite{BroMcC58PrimeIdlsNonassoc, Kaw74PrimeIdlsLie}.
Concretely, a Lie ideal $L$ in a \ca{} $A$ is said to be \emph{prime} if $L \neq A$ and if for all Lie ideals $K_1,K_2 \subseteq A$ with $[K_1,K_2] \subseteq L$ we have $K_1 \subseteq L$ or $K_2 \subseteq L$.
Further, a Lie ideal $L \subseteq A$ is \emph{semiprime} if for every Lie ideal $K \subseteq A$ with $[K,K] \subseteq L$ we have $K \subseteq L$.

\medskip

In this paper we describe prime and semiprime Lie ideals in a \ca{}~$A$ by establishing a correspondence with certain prime and semiprime two-sided ideals.
More specifically, if $L$ is a subspace in $A$, we consider its \emph{full normalizer subspace}
\[
T(L) := \big\{ a \in A \colon [A,a] \subseteq L \big\}.
\]
This construction was implicitly introduced by Herstein in \cite{Her55LieJordanSimpleRing}, and it played a key role in the study of Lie ideals in \cite{BreKisShu08LieIdeals}, where it is denoted by $N(L)$.

If $I$ is a semiprime two-sided ideal, we show in \autoref{prp:FullNormalizer} that $T(I)$ is a semiprime Lie ideal, and every semiprime Lie ideal arises this way; see \autoref{prp:ArisingFromIdeal}.
A crucial ingredient to obtain these results is the theory of powers of Dixmier ideals which was recently developed by the authors in \cite{GarKitThi23arX:SemiprimeIdls}.
We obtain a surjective assignment
\[
\big\{ \text{semiprime, two-sided ideals in $A$} \big\}
\xrightarrow{I \mapsto T(I)}
\big\{ \text{semiprime Lie ideals in $A$} \big\}.
\]

In general, this assignment is not injective.
For example, if $A$ is commutative, then $T(I)=A$ for every two-sided ideal~$I$, and~$A$ is the only semiprime Lie ideal in~$A$.
To remedy this, for a given semiprime Lie ideal $L$ we consider the (unique) largest two-sided ideal~$I$ in~$A$ satisfying $I \subseteq L$.
Among other things, we show that this two-sided ideal is semiprime and it can be characterized by the property that $A/I$ contains no nonzero, commutative, two-sided ideals.
The following is \autoref{prp:Correspondence}.

\begin{thmIntro}
Let $A$ be a \ca.
There is a natural, bijective correspondence 
\[
\left\{\ \parbox{7cm}{ semiprime two-sided ideals $I$ in $A$ such that $A/I$ contains no nonzero, commutative ideal } \ \right\}
\xrightarrow{I \mapsto T(I)}
\left\{ \ \parbox{2.2cm}{ semiprime Lie ideals in $A$ } \ \right\}.
\]
This assignment restricts to a bijection between prime two-sided ideals~$I$ in~$A$ such that~$A/I$ contains no nonzero, commutative, two-sided ideal, and prime Lie ideals in~$A$.
\end{thmIntro}

If a \ca{} is generated by its commutators as a ring (for example, if it is unital and has no characters \cite{GarThi25GenByCommutators}), then quotients by semiprime two-sided ideals cannot contain nonzero, commutative two-sided ideals.
We deduce the following:

\begin{corIntro}[{\autoref{prp:GenByCommutators}}]
Let $A$ be a \ca{} that is generated by its commutators as a ring.
There is a natural, bijective correspondence 
\[
\big\{ \text{semiprime two-sided ideals in $A$} \big\}
\xrightarrow{I \mapsto T(I)}
\big\{ \text{semiprime Lie ideals in $A$} \big\}.
\]
This assignment restricts to a bijection between prime two-sided ideals and prime Lie ideals.
\end{corIntro}

We say that a Lie ideal $L$ in a \ca{} $A$ is \emph{fully noncentral} if there is no proper two-sided ideal $I$ such that the image of $L$ in $A/I$ is central (equivalently, $A$ is generated by $[A,L]$ is a two-sided ideal).
This concept was introduced by Chand and Robert in \cite{ChaRob23AutoContGrUnit} in the context of \ca{s} with respect to closed two-sided ideals, and generalized by the first and last named author together with Lee in \cite[Definition~A]{GarLeeThi25FullyNoncentral} to the algebraic setting used here.
The next result is \autoref{prp:CharFullyNoncentral}.

\begin{thmIntro}
Let $A$ be a \ca{} that is generated by its commutators as a ring.
Then a Lie ideal in $A$ is fully noncentral if and only if it is not contained in any prime Lie ideal.
\end{thmIntro}

\section{Characterization of prime and semiprime Lie ideals}
\label{sec:SemiprimeLie}

In this section, we prove a characterization of (semi)prime Lie ideals in a \ca{} in terms of certain (semi)prime two-sided ideals;
see \autoref{prp:Correspondence}.

\medskip

For a Lie ideal $L$ in a \ca\ $A$, we write $T(L)=\{a\in A\colon [A,a]\subseteq L\}$. 
The following easy observation can be extracted from \cite[Lemma~3]{Her55LieJordanSimpleRing}, and we include it for the convenience of the reader.

\begin{lma}
\label{lma:TLieIdeal}
Let $L$ be a Lie ideal in a \ca\ $A$. 
Then $[A,T(L)]\subseteq L$, and thus $T(L)$ is a Lie ideal in $A$. 
Moreover, $T(L)$ is a subalgebra of $A$.
\end{lma}
\begin{proof}
The first assertion is obvious. 
For the second, observe that $T(L)$ is clearly closed under scalar multiplication and addition, so it suffices to show that it is closed under products. 
Let $s,t\in T(L)$ and let $a\in A$. 
We will show that $[a,st]$ belongs to $L$. 
Note that 
\[
[a,st]=[as,t]+[ta,s]
\]
is the sum of two elements in $L$, and thus belongs to $L$.
Thus, $st \in T(L)$.
\end{proof}

We will use $Z(R)$ to denote the center of a ring $R$.

\begin{lma}
\label{prp:T-Ideal}
Let $I$ be a two-sided ideal in a \ca{} $A$, and let $\pi \colon A \to A/I$ denote the quotient map.
Then
\[
T(I) = \big\{ a \in A : \pi(a) \in Z(A/I) \big\}.
\]
\end{lma}
\begin{proof}
Given $a \in A$, by definition we have $a \in T(I)$ if and only if $\pi([A,a]) = \{0\}$.
Using that $\pi([A,a]) = [A/I,\pi(a)]$, we see that $a \in T(I)$ if and only if $\pi(a) \in Z(A/I)$.
\end{proof}

A ring $R$ is said to have \emph{no 2-torsion} if multiplication by 2 is an injective map $R\to R$. 
No complex algebra has 2-torsion, and we will in particular use this for the quotient~$A/I$ of a \ca\ $A$ by an arbitrary two-sided ideal $I$. 

\begin{prp}
\label{prp:FullNormalizer}
Let $I$ be a two-sided ideal in a \ca{} $A$.
The following statements hold:
\begin{enumerate}
\item 
If $I$ is semiprime, then~$T(I)$ is a semiprime Lie ideal in $A$.
\item 
If $I$ is prime and $A/I$ is nonabelian, then $T(I)$ is a prime Lie ideal in $A$.
\end{enumerate}
\end{prp}
\begin{proof}
(1)
Suppose that $I$ is semiprime, and let $\pi \colon A \to A/I$ denote the quotient map.
By \autoref{prp:T-Ideal}, we have 
$T(I) = \big\{ a \in A\colon \pi(a) \in Z(A/I) \big\}$, 
and thus $\pi(T(I))=Z(A/I)$.
To show that $T(I)$ is a semiprime Lie ideal, let $K \subseteq A$ be a Lie ideal such that $[K,K] \subseteq T(I)$.
Then
\[
\big[ \pi(K),\pi(K) \big]
= \pi\big( [K,K] \big)
\subseteq \pi(T(I))
= Z(A/I).
\]
Using that $A/I$ is a semiprime ring without $2$-torsion, we deduce from \cite[Lemma~1]{Her70LieStructure} that $\pi(K) \subseteq Z(A/I)$, and thus $K \subseteq T(I)$, as desired.

\medskip

(2) 
Assume that $I$ is prime and that $A/I$ is nonabelian.
Then $A/I \neq Z(A/I)$ and it follows from \autoref{prp:T-Ideal} that $T(I) \neq A$.
To show that $T(I)$ is a prime Lie ideal, let $K_1,K_2 \subseteq A$ be Lie ideals such that $[K_1,K_2] \subseteq T(I)$.
Then
\[
\big[ \pi(K_1),\pi(K_2) \big]
= \pi\big( [K_1,K_2] \big)
\subseteq \pi(T(I))
= Z(A/I).
\]
Using that $A/I$ is a prime ring without $2$-torsion, we deduce from \cite[Lemma~7]{LanMon72LieStrPrimeChar2} that either $\pi(K_1) \subseteq Z(A/I)$ or $\pi(K_2) \subseteq Z(A/I)$, and thus $K_1 \subseteq T(I)$ or $K_2 \subseteq T(I)$.
\end{proof}

We briefly recall some basic results about the theory of powers of Dixmier ideals developed in \cite{GarKitThi23arX:SemiprimeIdls}.

\begin{pgr}
\label{pgr:DixmierIdls}
Let $A$ be a \ca.
Following \cite[Definition~3.2]{GarKitThi23arX:SemiprimeIdls}, we say that a two-sided ideal $I$ in $A$ is a \emph{Dixmier ideal} if it is positively spanned, hereditary, and its positive part is strongly invariant.
If $I \subseteq A$ is a Dixmier ideal, then for every $s \in (0,\infty)$ there is a unique Dixmier ideal $I^s \subseteq A$ whose positive part $(I^s)_+$ is $\{a^s : a \in I_+\}$;
see Proposition~3.7 and Definition~3.8 in \cite{GarKitThi23arX:SemiprimeIdls}.
By \cite[Theorem~3.9]{GarKitThi23arX:SemiprimeIdls}, we have $I^s I^t = I^{s+t}$ for all $s,t\in(0,\infty)$, and we have $I^s \subseteq I^t$ whenever $s \geq t >0$.

By \cite[Proposition~4.1]{GarKitThi23arX:SemiprimeIdls}, the Dixmier ideals in $A$ form a complete sublattice of the lattice of two-sided ideals, and in particular every two-sided ideal~$I$ contains a largest Dixmier ideal, called the \emph{Dixmier interior} of $I$ and denoted $ \interior(I)$, and is contained in a smallest Dixmier ideal, called the \emph{Dixmier closure} of $I$ and denoted $\exterior(I)$.
By \cite[Theorem~4.4]{GarKitThi23arX:SemiprimeIdls}, we have
\[
\exterior(I)^{1+\varepsilon} 
\subseteq \interior(I)
\subseteq I
\subseteq \exterior(I)
\subseteq \interior(I)^{1-\varepsilon} 
\]
for every $\varepsilon>0$.
\end{pgr}

Given a two-sided ideal $I$ in a ring $R$, we use $\sqrt{I}$ to denote its semiprime closure, that is, the intersection of all prime ideals in $R$ containing $I$;
see \cite[Section~10]{Lam01FirstCourse2ed} for details.

\begin{lma}
\label{prp:IdealInLieIdeal}
Let $L$ be a semiprime Lie ideal in a \ca{} $A$, and let $I \subseteq A$ be a two-sided ideal such that $I \subseteq L$.
Then $\sqrt{I} \subseteq L$.
\end{lma}
\begin{proof}
We will use the theory of powers of Dixmier ideals;
see \autoref{pgr:DixmierIdls}.
Set $J := \interior(I)$, the Dixmier interior of $I$.
Using \cite[Theorem~3.9]{GarKitThi23arX:SemiprimeIdls} at the second step, we get
\[
\big[ J^{\frac{1}{2}},J^{\frac{1}{2}} \big]
\subseteq J^{\frac{1}{2}} J^{\frac{1}{2}}
= J 
\subseteq I
\subseteq L.
\]
Using that $L$ is a semiprime Lie ideal, we get $J^{\frac{1}{2}} \subseteq L$.
Repeating this argument for
\[
\big[ J^\frac{1}{4},J^\frac{1}{4} \big]
\subseteq J^\frac{1}{4} J^\frac{1}{4}
= J^{\frac{1}{2}} 
\subseteq L,
\]
we get $J^\frac{1}{4} \subseteq L$.
Successively, we deduce that $J^{\frac{1}{2^n}} \subseteq L$ for all $n \geq 1$.
Applying \cite[Theorem~5.9]{GarKitThi23arX:SemiprimeIdls} at the first step, and using that $J^{\frac{1}{n}} \subseteq J^{\frac{1}{2^n}}$ for each $n$, we conclude that
\[
\sqrt{I} 
= \bigcup_{n=1}^{\infty} J^{\frac{1}{n}} 
\subseteq \bigcup_{n=1}^{\infty} J^{\frac{1}{2^n}} 
\subseteq L.\qedhere
\]
\end{proof}

Part~(1) of~\autoref{prp:FullNormalizer} shows that if $I$ is a semiprime two-sided ideal in a \ca, then its full normalizer is a semiprime Lie ideal.
The next result implies that all semiprime Lie ideals arise this way.
This will be made more precise in \autoref{prp:RangeOfIdeals} and \autoref{prp:Correspondence} below.

\begin{lma}
\label{prp:ArisingFromIdeal}
Let $L$ be a semiprime Lie ideal in a \ca{} $A$, and let $I_0$ and $I_{\mathrm{min}}$ denote the two-sided ideals in $A$ generated by $[L,L]$ and $[A,L]$, respectively.
Then
\[
I_0 \subseteq I_{\mathrm{min}}\subseteq \sqrt{I_{\mathrm{min}}} = \sqrt{I_0}, \andSep 
L = T(\sqrt{I_{\mathrm{min}}}). 
\]
\end{lma}
\begin{proof}
The inclusions $I_0 \subseteq I_{\mathrm{min}}\subseteq \sqrt{I_{\mathrm{min}}}$ are clear.
We will show that $L = T(\sqrt{I_0})$ and that $\sqrt{I_0} = \sqrt{I_{\mathrm{min}}}$.

We first verify $T(\sqrt{I_0}) \subseteq L$.
It is well-known that $I_0$ is contained in $L+L^2$;
see, for example, \cite[Lemma~1.1]{Rob16LieIdeals} or \cite[Lemma~2.1(ii)]{Lee22AddSubgpGenNCPoly}.
(See also \cite[Lemma~3.2]{GarThi25GenByCommutators} for a quantitative version.)
Using that $[a,bc] = [ab,c] + [ca,b]$ for all $a,b,c \in A$, we see that $[A,L^2] \subseteq [A,L]$;
see also, for example, \cite[Lemma~2.1(i)]{Lee22AddSubgpGenNCPoly}.
Combining these facts, we get
\[
[I_0,I_0]
\subseteq [A,I_0]
\subseteq [A,L+L^2]
\subseteq [A,L]
\subseteq L.
\]
Since $L$ is a semiprime Lie ideal, it follows that $I_0 \subseteq L$.
Applying \autoref{prp:IdealInLieIdeal}, we get $\sqrt{I_0} \subseteq L$.
It follows that
\[
\big[ T(\sqrt{I_0}),T(\sqrt{I_0}) \big] 
\subseteq \big[ A,T(\sqrt{I_0}) \big]
\subseteq \sqrt{I_0}
\subseteq L.
\]
Using that $T(\sqrt{I_0})$ is a Lie ideal by \autoref{lma:TLieIdeal}, we deduce that $T(\sqrt{I_0}) \subseteq L$.

\medskip

We now proceed to show that $L \subseteq T(\sqrt{I_0})$.
Let $\pi \colon A \to A/\sqrt{I_0}$ denote the quotient map.
Since $[L,L] \subseteq I_0$, we have
\[
\big[ \pi(L), \pi(L) \big]
= \pi\big( [L,L] \big)
= \{0\}.
\]
Thus, $\pi(L)$ is a Lie ideal in $A/\sqrt{I_0}$ satisfying $[\pi(L),\pi(L)] \subseteq Z(A/\sqrt{I_0})$.
Since $A/\sqrt{I_0}$ is a semiprime ring without $2$-torsion, we deduce from \cite[Lemma~1]{Her70LieStructure} that $\pi(L) \subseteq Z(A/\sqrt{I_0})$. Since $\pi^{-1}(Z(A/\sqrt{I_0})) = T(\sqrt{I_0})$ by \autoref{prp:T-Ideal}, we deduce that
$L \subseteq T(\sqrt{I_0})$, as desired.

\medskip

Finally, let us see that $\sqrt{I_0}=\sqrt{I_{\mathrm{min}}}$.
The inclusion `$\subseteq$' holds since $I_0 \subseteq I_{\mathrm{min}}$.
On the other hand, since $L \subseteq  T(\sqrt{I_0})$, we have $[A,L] \subseteq \sqrt{I_0}$.
Since $I_{\mathrm{min}}(L)$ is the 2-sided ideal generated by $[A,L]$ and $\sqrt{I_0}$ is a two-sided ideal, we get $I_{\mathrm{min}}\subseteq \sqrt{I_0}$, and thus $\sqrt{I_{\mathrm{min}}} \subseteq \sqrt{I_0}$.
\end{proof}

\begin{lma}
\label{prp:Imax}
Let $L$ be a linear subspace in a \ca{} $A$.
Then there exists a (unique) largest two-sided ideal $I_{\mathrm{max}}$ in $A$ that is contained in $L$.
\end{lma}
\begin{proof}
Set 
\[
I_0(L)=\bigcup \big\{ I : \text{$I$ is a two-sided ideal in $A$ with $I \subseteq L$} \big\}.
\]
Then clearly $I_0(L)\subseteq L$, and it is then easy to see that $aI_0(L),I_0(L)a\subseteq I_0(L)$ for all $a\in A$.
Setting $\Imax(L):=\linspan I_0(L)$, it follows that $\Imax(L)$ is a two-sided ideal in $A$.
Since $L$ is a linear subspace, we have $\Imax(L) \subseteq L$, and it clearly is the largest two-sided ideal contained in $L$.
\end{proof}

\begin{thm}
\label{prp:RangeOfIdeals}
Let $L$ be a semiprime Lie ideal in a \ca{} $A$, let $\Imin$ denote the two-sided ideal in~$A$ generated by $[A,L]$, and let $\Imax$ denote the largest two-sided ideal in~$A$ that is contained in $L$ (which exists by \autoref{prp:Imax}).
Then
\[
L = T(\Imin) = T(\Imax).
\]
Further, a two-sided $I \subseteq A$ satisfies $L=T(I)$ if and only if $\Imin \subseteq I \subseteq \Imax$.

The two-sided ideal $\Imax$ is semiprime.
It follows that a Lie ideal in $A$ is semiprime if and only if it has the form $T(J)$ for a semiprime two-sided ideal~$J$ in~$A$.
\end{thm}
\begin{proof}
We first show that $L = T(\Imin)$.
Since $[A,L] \subseteq \Imin$, the inclusion $L \subseteq T(\Imin)$ always holds.
By \autoref{prp:ArisingFromIdeal}, we have $L=T(\sqrt{\Imin})$, and since $\Imin \subseteq \sqrt{\Imin}$, we get
\[
T(\Imin) 
\subseteq T(\sqrt{\Imin})
= L.
\]

Next, we show that $L = T(\Imax)$.
Since $\Imin \subseteq T(\Imin) = L$, and since $\Imax$ is the largest two-sided ideal in $A$ that is contained in $L$, we get $\Imin \subseteq \Imax$, and thus $L=T(\Imin) \subseteq T(\Imax)$.
Conversely, we have
\[
[T(\Imax),T(\Imax)]
\subseteq [A,T(\Imax)]
\subseteq \Imax
\subseteq L,
\]
and since $L$ is a semiprime Lie ideal, we get $T(\Imax) \subseteq L$, as desired.

\medskip

Next, let $I$ be a two-sided ideal in $A$.
If $\Imin \subseteq I \subseteq \Imax$, then
\[
L = T(\Imin) \subseteq T(I) \subseteq T(\Imax) = L.
\]
and thus $L=T(I)$.
Conversely, assume that $L=T(I)$.
Since $I \subseteq T(I) \subseteq L$, we get $I \subseteq \Imax$.
On the other hand, we have $[A,L] \subseteq I$ and thus $\Imin \subseteq I$.

Finally, we show that $\Imax$ is semiprime.
Equivalently, we will show that $I_{\mathrm{max}}=\sqrt{I_{\mathrm{max}}}$. 
We have $\Imax \subseteq L$, and thus $\sqrt{\Imax} \subseteq L$ by \autoref{prp:IdealInLieIdeal}.
Since $\Imax$ is the largest two-sided idea in $A$ that is contained in $L$, we get $\sqrt{\Imax} \subseteq \Imax$ and so $\Imax = \sqrt{\Imax}$.
\end{proof}

\autoref{prp:RangeOfIdeals} naturally raises the following question, which, as we will see in \autoref{sec:SemiprimeCommIdl}, is closely related to the question of whether the commutator ideal of a \ca{} is semiprime.

\begin{qst}
\label{qst:ArisingFromIdeal}
Given a semiprime Lie $L$ ideal in a \ca{} $A$, is the two-sided ideal generated by $[A,L]$ automatically semiprime?
\end{qst}

By \autoref{prp:RangeOfIdeals}, assigning to a two-sided ideal $I \subseteq A$ its full normalizer subspace $T(I)$ defines a surjective map from semiprime two-sided ideals in~$A$ to semiprime Lie ideals in~$A$.
However, this map is not necessarily injective.
For example, if $A$ is commuative, then $T(I)=A$ for every two-sided ideal $I \subseteq A$.
More specifically, we have:

\begin{cor}
\label{prp:UniqueIdl}
Let $L$ be a semiprime Lie ideal in a \ca{} $A$, let $\Imin$ denote the two-sided ideal in~$A$ generated by $[A,L]$, and let $\Imax$ denote the largest two-sided ideal in~$A$ that is contained in $L$ (as in \autoref{prp:RangeOfIdeals}).
Then:
\begin{enumerate}
\item
There is a unique two-sided ideal $I$ in $A$ satisfying $L=T(I)$ if and only if $\Imin=\Imax$.
\item
There is a unique semiprime two-sided ideal $I$ in $A$ satisfying $L=T(I)$ if and only if $\sqrt{\Imin}=\Imax$.
\end{enumerate}
\end{cor}

We will see in \autoref{prp:UniqueRealizingIdeal} that $\Imin=\Imax$ for \ca{s} that are generated by their commutators.

In general, the map $I \mapsto T(I)$ defines a natural bijection onto the family of semiprime Lie ideals if we restrict to those semiprime two-sided ideals of the form $I_{\mathrm{max}}$, namely those that arise as the largest two-sided ideal contained in some semiprime Lie ideal.
The next result shows that such ideals can be abstractly characterized as the semiprime two-sided ideals $I$ such that $A/I$ contains no nonzero, commutative two-sided ideals.

\begin{thm}
\label{prp:Correspondence}
Let $A$ be a \ca{}.
Then the assignment $I \mapsto T(I)$ defines a natural bijection between the following sets:
\begin{enumerate}
\item[(a)] 
Semiprime two-sided ideals $I$ in $A$ such that $A/I$ contains no nonzero, commutative two-sided ideal.
\item[(b)] 
Semiprime Lie ideals in~$A$.
\end{enumerate}
The inverse of this assignment maps a semiprime Lie ideal $L$ in~$A$ to the largest two-sided ideal $I(L)$ of $A$ contained in $L$.

Further, the above bijection restricts to a bijection between the following sets: 
\begin{enumerate}
\item[(a')] 
Prime two-sided ideals $I$ in $A$ such that $A/I$ contains no nonzero, commutative two-sided ideal.
\item[(b')] 
Prime Lie ideals in~$A$.
\end{enumerate}
\end{thm}
\begin{proof}
Let $\mathcal{I}_{\mathrm{SP}}$ denote the family of semiprime two-sided ideals in $A$, and let $\mathcal{L}_{\mathrm{SP}}$ denote the family of semiprime Lie ideals in $A$.
We have seen in \autoref{prp:RangeOfIdeals} that $I \mapsto T(I)$ defines a surjective map
\[
T \colon \mathcal{I}_{\mathrm{SP}} \to \mathcal{L}_{\mathrm{SP}}.
\]
Further, for each $L \in \mathcal{L}_{\mathrm{SP}}$, the two-sided ideal $I_{\mathrm{max}}(L)$ is semiprime and satisfies $L=T(I_{\mathrm{max}}(L))$.
Thus, $T$ defines a bijection
\[
\mathcal{I}_{\mathrm{max}} := \big\{ I_{\mathrm{max}}(L)\colon L \in \mathcal{L}_{\mathrm{SP}} \big\} 
\to \mathcal{L}_{\mathrm{SP}}.
\]
We proceed to show that a two-sided ideal belongs to $\mathcal{I}_{\mathrm{max}}$ if and only if it satisfies the conditions listed in (a).

\medskip

\textbf{Claim 1:}
\emph{Given a semiprime Lie ideal $L \subseteq A$, the two-sided ideal $I_{\mathrm{max}}(L)$ satisfies the conditions listed in (a).}
First, $I_{\mathrm{max}}(L)$ is semiprime by \autoref{prp:RangeOfIdeals}.
Next, to see that $A/I_{\mathrm{max}}(L)$ contains no nonzero, commutative two-sided ideal, let $K_0 \subseteq A/I_{\mathrm{max}}(L)$ be a commutative two-sided ideal.
Let $\pi \colon A \to A/I_{\mathrm{max}}(L)$ denote the quotient map, and set $K := \pi^{-1}(K_0)$.
Then
\[
\pi( [K,K] ) = [K_0, K_0] = \{0\}.
\]

Since $A/I_{\mathrm{max}}(L)$ is a semiprime ring without $2$-torsion, it follows from \cite[Lemma~1]{Her70LieStructure} that $K_0 \subseteq Z(A/I_{\mathrm{max}}(L))$, and using \autoref{prp:T-Ideal} we get
\[
K \subseteq \pi^{-1}(Z(A/I_{\mathrm{max}}(L))) = T(I_{\mathrm{max}}(L)).
\]
Then
\[
[K,K] 
\subseteq [A,T(I_{\mathrm{max}}(L))]
\subseteq I_{\mathrm{max}}(L)
\subseteq L,
\]
and since $L$ is a semiprime Lie ideal, we get $K \subseteq L$.
By maximality of $I_{\mathrm{max}}(L)$, we get $K\subseteq I_{\mathrm{max}}(L)$, that is, $K_0=\{0\}$.

\medskip

\textbf{Claim 2:}
\emph{Given a semiprime two-sided ideal $I \subseteq A$ satisfying the conditions listed in~(a), we have $I_{\mathrm{max}}(T(I)) = I$.}
First, we have $I \subseteq T(I)$, and consequently $I \subseteq I_{\mathrm{max}}(T(I))$.
It remains to show that $I_{\mathrm{max}}(T(I)) \subseteq I$.
Let $\pi \colon A \to A/I$ denote the quotient map.
By \autoref{prp:T-Ideal}, we have $\pi(T(I)) \subseteq Z(A/I)$.
It follows that $\pi(I_{\mathrm{max}}(T(I)))$ is a two-sided ideal in~$A/I$ contained in the center of $A/I$, and hence a commutative two-sided ideal in $A/I$.
By assumption, this implies that $\pi(I_{\mathrm{max}}(T(I)))=\{0\}$, that is, $I_{\mathrm{max}}(T(I)) \subseteq I$, as desired.

\medskip

We have shown that the assignment $I\mapsto T(I)$ is a bijection (with inverse given by $L\mapsto I_{\mathrm{max}}(L)$) between the sets described in~(a) and~(b).
Finally, let us show that these assignments restrict to bijections between the ideals in (a') and (b').

In one direction, if $I$ is a prime two-sided ideal in~$A$ as in (a'), then $A/I$ is nonabelian, and it follows from \autoref{prp:FullNormalizer} that $T(I)$ is a prime Lie ideal.

Conversely, assume that $L$ is a prime Lie ideal.
We have already seen that~$I_{\mathrm{max}}(L)$ is a semiprime ideal in $A$ as in (a) satisfying $L = T(I_{\mathrm{max}}(L))$.
Since $L \neq A$, it follows that $I_{\mathrm{max}}(L) \neq A$.
To show that $I_{\mathrm{max}}(L)$ is a prime two-sided ideal, let $I_1,I_2 \subseteq A$ be two-sided ideals such that $I_1I_2 \subseteq I_{\mathrm{max}}(L)$.
Since $I_{\mathrm{max}}(L)$ is an ideal, we deduce that $I_2I_1I_2I_1 \subseteq I_{\mathrm{max}}(L)$.
Using that $I_{\mathrm{max}}(L)$ is semiprime, it follows that $I_2I_1 \subseteq I_{\mathrm{max}}(L)$.
We get
\[
[I_1,I_2] 
\subseteq I_1I_2 + I_2I_1 
\subseteq I_{\mathrm{max}}(L) 
\subseteq T(I_{\mathrm{max}}(L))
= L.
\]
Since $L$ is a prime Lie ideal, we get $I_1 \subseteq L$ or $I_2 \subseteq L$.
By maximality of $I_{\mathrm{max}}(L)$, it follows that $I_1 \subseteq I_{\mathrm{max}}(L)$ or $I_2 \subseteq I_{\mathrm{max}}(L)$, as desired.
\end{proof}

\section{Lie ideals in \texorpdfstring{$C^*$}{C*}-algebras with semiprime commutator ideal}
\label{sec:SemiprimeCommIdl}

In this section, we study \autoref{qst:ArisingFromIdeal}:
Given a semiprime Lie ideal $L$ in a \ca{} $A$, is the two-sided ideal generated by $[A,L]$ automatically semiprime?
Note that $A$ is always a semiprime Lie ideal itself, and thus a particular instance of this question is whether the commutator ideal (that is, the two-sided ideal generated by $[A,A]$) is semiprime;
see \autoref{qst:SemiprimeCommIdl}.
Surprisingly, we will see that \autoref{qst:ArisingFromIdeal} reduces to this special case:
Given a \ca{} $A$, the two-sided ideal generated by $[A,L]$ is semiprime for every semiprime Lie ideal $L \subseteq A$ if (and only if) the two-sided ideal generated by $[A,A]$ is semiprime;
see \autoref{prp:SemiprimeCommIdl}.
The latter is trivially the case if the commutator ideal is all of $A$, a situation that is studied in more detail in \autoref{sec:GenByComm}.
More generally, it holds whenever the commutator ideal is closed, which is automatic, for example, in von Neumann algebras;
see \autoref{exa:VNA}.

\medskip

The next result implies in particular that for every two-sided ideal $I$ in a \ca{}, the two-sided ideal $I\sqrt{I}$ is Dixmier, since it is the supremum of the Dixmier ideals $\interior(I)^{1+\varepsilon}$ for $\varepsilon>0$.
The theory of powers of Dixmier ideals is briefly recalled in \autoref{pgr:DixmierIdls}.

\begin{lma}
\label{prp:WrappingIdl}
Let $I$ be a two-sided ideal in a \ca{} $A$, let~$J=\interior(I)$ denote its Dixmier interior, and let~$K=\exterior(I)$ denote its Dixmier closure.
Then
\[
\sqrt{J} = \sqrt{I} = \sqrt{K}, \andSep
J\sqrt{J} = I\sqrt{I} = K\sqrt{K} 
= \bigcup_{\varepsilon>0} J^{1+\varepsilon}
= \bigcup_{\varepsilon>0} K^{1+\varepsilon}.
\]
Further, we have $I\sqrt{I} = \sqrt{I}I$.
\end{lma}
\begin{proof}
By \cite[Theorem~5.9]{GarKitThi23arX:SemiprimeIdls}, we have
\[
\sqrt{J} = \sqrt{I} = \sqrt{K}
= \bigcup_{\varepsilon>0} J^{\varepsilon}
= \bigcup_{\varepsilon>0} K^{\varepsilon}.
\]
and using the basic results about Dixmier ideals recalled in \autoref{pgr:DixmierIdls}, 
we deduce that
\[
J\sqrt{J} 
= J \bigcup_{\varepsilon>0} J^{\varepsilon}
= \bigcup_{\varepsilon>0} J^{1+\varepsilon}.
\]
Analogously, we see that $K\sqrt{K}  
= \bigcup_{\varepsilon>0} K^{1+\varepsilon}$.
The inclusions $J\sqrt{J} \subseteq I\sqrt{I} \subseteq K\sqrt{K}$ are clear.
For every $\varepsilon>0$, we have $K^{1+\varepsilon} \subseteq J$, and 
using \cite[Theorem~3.9]{GarKitThi23arX:SemiprimeIdls}
at the first step, we get
\[
K^{1+2\varepsilon} 
= K^{1+\varepsilon} K^{\varepsilon} 
\subseteq J K^{\varepsilon}.
\]
It follows that
\[
K\sqrt{K}  
= \bigcup_{\varepsilon>0} K^{1+\varepsilon}
= \bigcup_{\varepsilon>0} K^{1+2\varepsilon}
\subseteq \bigcup_{\varepsilon>0} JK^{\varepsilon}
= J \bigcup_{\varepsilon>0} K^{\varepsilon}
= J \sqrt{K}
= J \sqrt{J}.
\]
An analogous argument shows that
\[
\sqrt{I}I 
= \sqrt{K}K
= \bigcup_{\varepsilon>0} K^{1+\varepsilon}
= K\sqrt{K}
= I\sqrt{I},
\]
as desired.
\end{proof}

Given a \ca{} $A$ we use $\widetilde{A}$ to denote its minimal unitization.
Then the (algebraic\footnote{If we were interested in the \emph{closed} two-sided ideal that $V$ generates, then this simply equal to $\overline{AVA}$ since $A$ has an approximate identity, and there is no need to 
consider the unitization of $A$.}) two-sided ideal of $A$ generated by a subspace $V \subseteq A$ is $\widetilde{A}V\widetilde{A}$.
If $A$ is nonunital, then $AVA$ is a two-sided ideal that may be strictly smaller than $\widetilde{A}V\widetilde{A}$.
If $L \subseteq A$ is a Lie ideal, then $\widetilde{A}L\widetilde{A} = \widetilde{A}L = L\widetilde{A}$.
For the commutator subspace $[A,A]$, we have $\widetilde{A}[A,A]\widetilde{A} = A[A,A]A = A[A,A] = [A,A]A$ by \cite[Proposition~3.2]{GarThi24PrimeIdealsCAlg}.

The next result shows that for a semiprime Lie ideal $L$ in a \ca{} $A$, the two-sided ideals generated by $[L,L]$ and by $[A,L]$ are `almost equal'.
This suggests that they may in fact be equal, and we know of no examples where this is not the case. 
We verify the equality of these two ideals in the case that the commutator ideal is semiprime (\autoref{prp:SemiprimeCommIdl}), but it remains open in general.

\begin{prp}
\label{prp:ISqrtI}
Let $L \subseteq A$ be a semiprime Lie ideal in a \ca{} $A$, let~$I_0$ and~$I_{\mathrm{min}}$ denote the two-sided ideals in $A$ generated by $[L,L]$ and $[A,L]$ respectively, and let~$K=\exterior(I_{\mathrm{min}})$ denote the Dixmier closure of $I_{\mathrm{min}}$.
Then
\begin{align*}\tag{3.1}
\label{eq:ISqrtI}
I_0 \sqrt{I_{\mathrm{min}}} &= I_{\mathrm{min}} \sqrt{I_{\mathrm{min}}} = [A,A]\sqrt{I_{\mathrm{min}}} = \widetilde{A}[A,A]\widetilde{A}\sqrt{I_{\mathrm{min}}}= \bigcup_{\varepsilon>0} K^{1+\varepsilon}\\
&\subseteq I_0 \subseteq I_{\mathrm{min}} \subseteq K.
\end{align*}
\end{prp}
\begin{proof}
We abbreviate $I_{\mathrm{min}}$ to $I$ throughout this proof. 
We first show the first three equalities in \eqref{eq:ISqrtI}.
The inclusions $I_0 \sqrt{I} \subseteq I \sqrt{I}$ and $[A,A]\sqrt{I} \subseteq \widetilde{A}[A,A]\widetilde{A}\sqrt{I}$ are clear.
Further, since $[A,L]$ is a Lie ideal, we have $I = [A,L]\widetilde{A}$ and therefore
\[
I \sqrt{I} 
= [A,L]\widetilde{A}\sqrt{I} 
\subseteq [A,A]\sqrt{I}.
\]

It remains to show that $\widetilde{A}[A,A]\widetilde{A}\sqrt{I}$ is contained in $I_0 \sqrt{I}$.
Using that $[a,b]x = [a,bx] - b[a,x]$ for $a,b,x \in A$, we get
\begin{align*}\tag{3.2}\label{eq:3.2}
\widetilde{A} [A,A]\sqrt{I} 
\subseteq \widetilde{A} [A,A\sqrt{I}] + \widetilde{A}A[A,\sqrt{I}]
\subseteq \widetilde{A} [A,\sqrt{I}].
\end{align*}
Using at the first step that $\sqrt{I}$ is idempotent (that is, $\sqrt{I}=\sqrt{I}\sqrt{I}$) by \cite[Theorem~5.2]{GarKitThi23arX:SemiprimeIdls}, using at the second step that $[a,xy]=[ax,y]+[ya,x]$ for $a,x,y \in A$, and using at the fourth step that $\sqrt{I}\subseteq L$ by \autoref{prp:ArisingFromIdeal}, we get 
\begin{align*}\tag{3.3}\label{eq:3.3}
[A,\sqrt{I}]
= [A,\sqrt{I}\sqrt{I}]
\subseteq [A\sqrt{I},\sqrt{I}] + [\sqrt{I}A,\sqrt{I}]
\subseteq [\sqrt{I},\sqrt{I}]
\subseteq \widetilde{A} [L,L] \widetilde{A}
= I_0.
\end{align*}
Using again that $\sqrt{I}$ is idempotent at the first step, we get 
\[
\widetilde{A}[A,A]\widetilde{A}\sqrt{I}
\subseteq \widetilde{A}[A,A]\sqrt{I}\sqrt{I}
\stackrel{\eqref{eq:3.2}}{\subseteq} \widetilde{A} [A,\sqrt{I}]\sqrt{I}
\stackrel{\eqref{eq:3.3}}{\subseteq} \widetilde{A} I_0 \sqrt{I}
\subseteq I_0 \sqrt{I}.
\]

We have shown the first three equalities in \eqref{eq:ISqrtI}.
The fourth follows from \autoref{prp:WrappingIdl}.
Finally, the inclusions $I_0 \sqrt{I} \subseteq I_0 \subseteq I \subseteq K$ are clear.
\end{proof}

The next result gives a positive answer to \autoref{qst:ArisingFromIdeal} for \ca{} whose commutator ideal is semiprime.

\begin{thm}
\label{prp:SemiprimeCommIdl}
Let $A$ be a \ca{} whose commutators generate a semiprime two-sided ideal.
Let $L \subseteq A$ be a semiprime Lie ideal, and let~$I_0$ and~$I_{\mathrm{min}}$ denote the two-sided ideals in $A$ generated by $[L,L]$ and $[A,L]$ respectively.
Then
\[
I_0 = I_{\mathrm{min}} = \sqrt{I_{\mathrm{min}}}, \andSep
L = T(I_{\mathrm{min}}).
\]
In particular, $I_{\mathrm{min}}$ is a semiprime two-sided ideal.
\end{thm}
\begin{proof}
We abbreviate $I_{\mathrm{min}}$ to $I$ throughout this proof. 
Let $C$ denote the commutator ideal of~$A$.
We have $I \subseteq C$.
Since $C$ is semiprime, we have $C=C^2$ by \cite[Theorem~5.2]{GarKitThi23arX:SemiprimeIdls}, and thus $I \subseteq C^2$.
Using that $C\sqrt{I} = I\sqrt{I}$ by \autoref{prp:ISqrtI}, and using that $I\sqrt{I} = \sqrt{I}I$ by \autoref{prp:WrappingIdl}, we get
\[
I\sqrt{I}
\subseteq C^2\sqrt{I}
= CI\sqrt{I}
= C\sqrt{I}I
= I\sqrt{I}I
= I^2\sqrt{I}
\subseteq I^2.
\]

Let~$K$ denote the Dixmier closure of $I$.
Using \autoref{prp:ISqrtI}, we get
\[
K^{\frac{3}{2}}
\subseteq \bigcup_{\varepsilon>0} K^{1+\varepsilon}
= I\sqrt{I}
\subseteq I^2
\subseteq K^2.
\]
It follows from \cite[Lemma~5.1]{GarKitThi23arX:SemiprimeIdls} that $K$ is idempotent.
Using \autoref{prp:ISqrtI} again, it follows that
\[
K 
= K^2 
\subseteq \bigcup_{\varepsilon>0} K^{1+\varepsilon}
\subseteq I_0 
\subseteq I \subseteq K,
\]
and thus $I_0=I=I^2$.
Applying \cite[Theorem~5.2]{GarKitThi23arX:SemiprimeIdls}, it follows that $I$ is semiprime, and so $I=\sqrt{I}$.
Using \autoref{prp:ArisingFromIdeal}, we get
\[
T(I)=T(\sqrt{I})=L,
\]
as desired.
\end{proof}

The following is a natural question.

\begin{qst}
\label{qst:SemiprimeCommIdl}
Is the commutator ideal of every \ca\ semiprime?
\end{qst}

For von Neumann algebras, we show in the next proposition that the answer is positive.

\begin{prp}
\label{exa:VNA}
Let $M$ be a von Neumann algebra.
Then the commutator ideal in $M$ is semiprime.
\end{prp}
\begin{proof}
It is a standard fact that there is a largest central projection $z \in M$ such that $zM$ is commutative.
This yields a decomposition $M=zM\oplus(1-z)M$, and $(1-z)M$ is the commutator ideal.
In particular, the commutator ideal in $M$ is closed and therefore semiprime.
Thus, \autoref{prp:SemiprimeCommIdl} applies and shows that for every semiprime Lie ideal $L \subseteq M$, the two-sided ideal $M[M,L]M$ is semiprime and we have $L=T\big(M[M,L]M\big)$.
\end{proof}

\section{Lie ideals in \texorpdfstring{$C^*$}{C*}-algebras that are generated by commutators}
\label{sec:GenByComm}

In this section, we specialize further to \ca{s} that are generated by their commutators as a (not necessarily closed) two-sided ideal. 
For such a \ca{} $A$, we show that for every semiprime Lie ideal $L \subseteq A$ there is a unique (automatically semiprime) two-sided ideal whose full normalizer subspace is $L$;
see \autoref{prp:UniqueRealizingIdeal}.
It follows that there is a natural bijection between semiprime two-sided ideals in~$A$ and semiprime Lie ideals in~$A$;
see \autoref{prp:GenByCommutators}.

\begin{lma}
\label{prp:SemiprimeQuotGenByCommutators}
Let $A$ be a \ca{} that is generated by its commutators as an ideal, and let $I$ be a semiprime two-sided ideal in $A$.
Then  $A/I$ has no nonzero, commutative two-sided ideals.
\end{lma}
\begin{proof}
Let $J \subseteq A/I$ be a commutative two-sided ideal.
We show below that~$J$ is a nilpotent ideal in~$A/I$, specifically that $J^3 = \{0\}$.
Since $A/I$ is a semiprime ring, it contains no nonzero nilpotent ideals \cite[Proposition~10.16]{Lam01FirstCourse2ed}, and thus $J = \{0\}$.

Let $x,y,z \in J$.
Given $a,b \in A/I$, we have $xa, yb \in J$, which implies that they commute.
Using this at the first step, using at the third step that $xa$ and $y$ commute and that $yb$ and $x$ commute, and using at the fourth step that $x$ and $y$ commute, we get
\[
0
= [xa,yb]
= xayb - ybxa
= yxab - xyba
= xyab - xyba
= xy[a,b].
\]

Since $A$ is generated by its commutators as a ring by \cite[Theorem~4.1]{GarThi25GenByCommutators}, so is~$A/I$.
By writing $z$ as a finite sum of products of commutators and using the above, we conclude that $xyz=0$, as desired.
\end{proof}

With the notation from \autoref{prp:RangeOfIdeals}, the next result shows that $\Imin(L)=\Imax(L)$ for every semiprime Lie ideal $L$ in a \ca{} that is generated by its commutators.

\begin{thm}
\label{prp:UniqueRealizingIdeal}
Let $A$ be a \ca{} that is generated by its commutators as an ideal, and let $L$ be a semiprime Lie ideal in~$A$.
Then the two-sided ideal generated by $[A,L]$ is the unique two-sided ideal $I \subseteq A$ satisfying $L=T(I)$.
\end{thm}
\begin{proof}
As in \autoref{prp:RangeOfIdeals}, let $\Imin$ denote the two-sided ideal in~$A$ generated by $[A,L]$, and let $\Imax$ denote the largest two-sided ideal in~$A$ that is contained in~$L$ (which exists by \autoref{prp:Imax}).
By assumption, the commutator ideal of~$A$ is~$A$ itself, and is thus semiprime.
Therefore, $\Imin$ is semiprime by \autoref{prp:SemiprimeCommIdl}.
Since $\Imax \subseteq L$ and since $[L,L] \subseteq \Imin$, we see that $\Imax/\Imin$ is an abelian two-sided ideal in~$A/\Imin$. 
Applying \autoref{prp:SemiprimeQuotGenByCommutators}, it follows that $\Imax/\Imin=\{0\}$, that is, $\Imin=\Imax$.
Now the result follows from \autoref{prp:UniqueIdl}.
\end{proof}

\begin{cor}
\label{prp:GenByCommutators}
Let $A$ be a \ca{} that is generated by its commutators as 
an ideal.
Then the assignment $I \mapsto T(I)$ defines a natural bijection between semiprime two-sided ideals in~$A$ and semiprime Lie ideals in~$A$.

This correspondence restricts to a natural bijection between prime two-sided ideals in~$A$ and prime Lie ideals in~$A$.
\end{cor}
\begin{proof}
Surjectivity of the assignment is shown in \autoref{prp:RangeOfIdeals}.
Injectivity holds by \autoref{prp:UniqueRealizingIdeal}.
\end{proof}

\begin{exa}
\label{exa:Simple}
Let $A$ be a unital, simple, nonabelian \ca{}.
Since the only two-sided ideals in $A$ are $\{0\}$ and~$A$ itself, and they are automatically semiprime, it follows from \autoref{prp:GenByCommutators} that the only semiprime Lie ideals in $A$ are
\[
T(\{0\}) = Z(A)=\mathbb{C}1_A, \andSep
T(A) = A.
\]
Moreover, the only prime Lie ideal in $A$ is $Z(A)=\mathbb{C}1_A$.

On the other hand, it follows from Herstein's description of Lie ideals in simple rings \cite{Her55LieJordanSimpleRing} that a subspace $L \subseteq A$ is a Lie ideal if and only if $L \subseteq Z(A)$ or $[A,A] \subseteq L$.
Further, if $A$ admits tracial states, then $[A,A] \neq A$ and $A$ contains Lie ideals that are not semiprime.
\end{exa}

\begin{rmk}
\label{rmk:Closed}
While every closed two-sided ideal in a \ca{} is well-known to be semiprime, the analog for Lie ideals is not true.
For example, if $A$ is a unital, simple, noncommutative \ca{} with a tracial state, then the closure of $[A,A]$ is a closed Lie ideal that is not semiprime.
\end{rmk}

\begin{exa}
\label{exa:ProInf}
Let $A$ be a unital, properly infinite \ca{}, and let $L \subseteq A$ be a Lie ideal.
Let $I := A[A,L]A$ be the two-sided ideal generated by $[A,L]$.
It was shown in \cite[Theorem~5.3]{Thi26LiePropInf} that $L$ is \emph{embraced by $I$}, that is, we have
\[
[A,I] \subseteq L \subseteq T([A,I]).
\]
Further, $I$ is the unique two-sided ideal of $A$ that embraces $L$.

Letting $\mathcal{I}$ denote the family of all two-sided ideals in $A$, it follows that the collection $\mathcal{L}$ of Lie ideals in $A$ decomposes as a disjoint union
\[
\mathcal{L} = 
\bigsqcup_{I \in \mathcal{I}} 
\big\{ L \subseteq A \text{ subspace} : [A,I] \subseteq L \subseteq T([A,I]) \big\};
\]
see \cite[Corollary~5.4]{Thi26LiePropInf}.
Which of these Lie ideals are (semi)prime?

Since $A$ is generated by its commutators as an ideal, \autoref{prp:GenByCommutators} applies and we see that a Lie ideal $L$ is (semi)prime if and only if $L=T(I)$ for some (semi)prime two-sided ideal $I$.

\textbf{Claim:} \emph{If $I$ is a two-sided ideal in $A$, then $T(I)=T([A,I])$.}
It suffices to verify that $T(I) \subseteq T([A,I])$.
Set $J=A[A,T(I)]A$. Since $[A,T(I)] \subseteq I$, we get $J \subseteq I$. 
Applying \cite[Theorem~5.3]{Thi26LiePropInf} for the Lie ideal $T(I)$, we see that $T(I) \subseteq T([A,J])$, and thus
\[
T(I) 
\subseteq T([A,J])
\subseteq T([A,I]).
\]

It follows that the collection $\mathcal{L}$ of Lie ideals in $A$ decomposes as a disjoint union
\[
\mathcal{L} = 
\bigsqcup_{I \in \mathcal{I}} 
\big\{ L \subseteq A \text{ subspace} : [A,I] \subseteq L \subseteq T(I) \big\}.
\]
Consequently, a Lie ideal in $A$ is (semi)prime if and only if it is the largest subspace embraced by some (semi)prime two-sideal ideal in $A$.
\end{exa}

\begin{thm}
\label{prp:CharFullyNoncentral}
Let $A$ be a \ca{} that is generated by its commutators as an ideal.
Then a Lie ideal in $A$ is fully noncentral if and only if it is not contained in any prime Lie ideal.
\end{thm}
\begin{proof}
Let $L$ be a Lie ideal in $A$.
For the forward direction, assume that $L$ is fully noncentral.
To reach a contradiction, assume that $K$ is a prime Lie ideal containing~$L$.
Let $I$ is the semiprime closure of the two-sided ideal generated by $[A,K]$. By \autoref{prp:ArisingFromIdeal}, we have $K=T(I)$.
Since $L$ is fully noncentral, the two-sided ideal generated by $[A,L]$ is all of $A$.
Since $[A,L] \subseteq [A,K]$, it follows that $I=A$ and so $K=A$, a contradiction.

Conversely, assume that $L$ is not fully noncentral.
Let $I$ denote the two-sided ideal generated by $[A,L]$.
By assumption, $I \neq A$.
Applying \cite[Theorem~A]{GarThi24PrimeIdealsCAlg}, we obtain a prime ideal $J \subseteq A$ such that $I \subseteq J$.
It follows that $[A,L] \subseteq I \subseteq J$, and therefore $L \subseteq T(J)$.
Note that $A/J$ is nonabelian since it is nonzero (since $J \neq A$) and since $A/J$ has no nonzero, commutative two-sided ideals by \autoref{prp:SemiprimeQuotGenByCommutators}.
Thus, it follows from \autoref{prp:FullNormalizer} that $T(J)$ is a prime Lie ideal.
\end{proof}


\providecommand{\bysame}{\leavevmode\hbox to3em{\hrulefill}\thinspace}
\providecommand{\noopsort}[1]{}
\providecommand{\mr}[1]{\href{http://www.ams.org/mathscinet-getitem?mr=#1}{MR~#1}}
\providecommand{\zbl}[1]{\href{http://www.zentralblatt-math.org/zmath/en/search/?q=an:#1}{Zbl~#1}}
\providecommand{\jfm}[1]{\href{http://www.emis.de/cgi-bin/JFM-item?#1}{JFM~#1}}
\providecommand{\arxiv}[1]{\href{http://www.arxiv.org/abs/#1}{arXiv~#1}}
\providecommand{\doi}[1]{\url{http://dx.doi.org/#1}}
\providecommand{\MR}{\relax\ifhmode\unskip\space\fi MR }
\providecommand{\MRhref}[2]{%
  \href{http://www.ams.org/mathscinet-getitem?mr=#1}{#2}
}
\providecommand{\href}[2]{#2}

\end{document}